\documentclass[12pt]{amsart}
\usepackage{amsmath,amssymb,amsbsy,amsfonts,latexsym,amsopn,amstext,cite,
                                               amsxtra,euscript,amscd,bm,mathabx,mathrsfs}
\usepackage{url}
\usepackage[colorlinks,linkcolor=blue,anchorcolor=blue,citecolor=blue,backref=page]{hyperref}
\usepackage{color}
\usepackage{graphics,epsfig}
\usepackage{graphicx}
\usepackage{float} 
\usepackage[english]{babel}
\usepackage{mathtools}
\usepackage{todonotes}
\usepackage{url}
\usepackage[colorlinks,linkcolor=blue,anchorcolor=blue,citecolor=blue,backref=page]{hyperref}

\usepackage[norefs,nocites]{refcheck}

\hypersetup{breaklinks=true}

\usepackage[norefs,nocites]{refcheck}
\usepackage[english]{babel}

\newtheorem{thm}{Theorem}
\newtheorem{lem}[thm]{Lemma}

\newtheorem{rem}[thm]{Remark}

\numberwithin{equation}{section}
\numberwithin{thm}{section}
\numberwithin{table}{section}

\def\squareforqed{\hbox{\rlap{$\sqcap$}$\sqcup$}}
\def\qed{\ifmmode\squareforqed\else{\unskip\nobreak\hfil
\penalty50\hskip1em\null\nobreak\hfil\squareforqed
\parfillskip=0pt\finalhyphendemerits=0\endgraf}\fi}

\def \balpha{\bm{\alpha}}
\def \bbeta{\bm{\beta}}

\def\cE{{\mathcal E}}

\def\cP{{\mathcal P}}

\def\cS{{\mathcal S}}

\def\cX{{\mathcal X}}

\def\inv{\mathrm{inv}} 
\def\sqrt{\mathrm{sqrt}}

\def \F {{\mathbb F}}

\def\\{\cr}
\def\({\left(}
\def\){\right)}
\def\fl#1{\left\lfloor#1\right\rfloor}
\def\rf#1{\left\lceil#1\right\rceil}

\def\mand{\qquad\mbox{and}\qquad}

\newcommand{\ov}{\overbar}
\newcommand{\overbar}[1]{\mkern 1.5mu\overline{\mkern-1.5mu#1\mkern-1.5mu}\mkern 1.5mu}

 \newcommand{\Mod}[1]{\ (\mathrm{mod}\ #1)}

\begin{document}

\title[Energy bounds, bilinear sums in function fields]
{Energy bounds, bilinear sums and their applications   in function fields}

 \author[C.~Bagshaw]{Christian Bagshaw}
 \address{School of Mathematics and Statistics, University of New South Wales.
 Sydney, NSW 2052, Australia}
 \email{c.bagshaw@unsw.edu.au}

 \author[I.~E.~Shparlinski]{Igor E. Shparlinski}
 \address{School of Mathematics and Statistics, University of New South Wales.
 Sydney, NSW 2052, Australia}
 \email{igor.shparlinski@unsw.edu.au}

 \begin{abstract}  
We obtain function field analogues of recent energy bounds on modular square roots and modular inversions and apply them to bounding some bilinear sums and to some questions regarding smooth and square-free polynomials in residue classes. 
\end{abstract}  


\maketitle

\tableofcontents

\section{Introduction}

\subsection{Motivation} 

Given a prime number $p$ we denote by $\F_p$ the field of $p$ elements, which we assume to be 
represented by the set $\{0, \ldots, p-1\}$ and for a positive integer $n <p$ we define  
the {\it additive energy  of modular square roots\/} as
\begin{align*}
\cE_p^{\sqrt}(n) =     \#\{(u,v, x,y) \in \F_p^4:~ u + v &= x+y, \\
&    u^2,v^2, x^2,y^2 \in \{1,\ldots, n\}  \}. 
\end{align*}  
Recently, several bounds on $\cE_p^{\sqrt}(n)$ have been established in~\cite{DKSZ,KSSZ,SSZ1}, which in particular imply
\begin{equation} 
\label{eq;Energy Sqrt}
\cE_p^{\sqrt}(n) \le   \min\left\{n^4/p + n^{5/2}, \,  n^{7/2}/p^{1/2}+ n^{2}   \right\} p^{o(1)}.
\end{equation} 
These energy bounds,  such as~\eqref{eq;Energy Sqrt},  have been used in~\cite{DKSZ,KSSZ,SSZ1} 
to estimate certain bilinear sums with square roots and also  to bounding sums of Sali{\'e} sums, improving 
that of Dunn and Zaharescu~\cite{DuZa}, which is based on a different approach. In turn, the  bounds of such bilinear sums  have
found several applications to such classical number theoretic topics as 
 asymptotic formulas for moments of some   half integral weight $L$-functions, 
distribution of modular square roots of primes and  variations of the  Erd\H{o}s-Odlyzko-S\'{a}rk\"{o}zy conjecture~\cite{EOS}, 
see~\cite{DKSZ, DuZa, KSSZ, SSZ1, SSZ2} for more details. 

Furthermore,  Heath-Brown~\cite[Page~368]{H-B} has given the  bound
\begin{equation} 
\label{eq;Energy Inv}
\cE_p^{\inv}(n) \le \(\frac{n^{7/2}}{p^{1/2}} + n^2\) p^{o(1)}
\end{equation} 
 on the 
 {\it additive energy  of modular inversions\/} 
  \begin{align*}
\cE_p^{\inv}(n) =     \#\{(u,v, x,y) \in & \{1,\ldots, n\}^4:\\
 &u^{-1} + v^{-1}  \equiv x^{-1} +y^{-1}  \pmod p \}. 
\end{align*}
We also note that Bourgain and  Garaev~\cite[Corollary~4]{BouGar} have given 
a different way to establish~\eqref{eq;Energy Inv} (which is 
the approach we use in this work).  Various applications on the bound~\eqref{eq;Energy Inv}
have also been given in~\cite{BouGar,H-B}. Additionally, it has been used in~\cite{MuShp,MSY} 
to study the distribution of  integers of prescribed arithmetic structure 
(such as smooth, square-free and square-full) in arithmetic progressions and in~\cite{Shp} to new bounds of bilinear sums with Kloosterman sums.

To give an example of such applications, 
we recall that an integer $n$ is called $y$-smooth if all prime divisors  $\ell \mid n$ 
satisfy $\ell \le y$.  Now, for a prime $p$, we define $M(p)$ to be the smallest number such that every residue   class modulo $p$ can be represented by a $p$-smooth  square-free integer 
not exceeding $M(p)$. 
Booker and Pomerance~\cite{BoPom} prove that for $p > 7$ this quantity is well-defined 
(that is,  $M(p) < \infty$), give the bound $M(p)=p^{O(\log p)}$ and conjecture  that $M(p) \leq p^{O(1)}$. This conjecture has 
been settled in~\cite{MuShp} in a more general form. 
Furthermore, for a real $\alpha>0$, both papers~\cite{BoPom} and~\cite{MuShp}
also define and discuss a more general quantity  $M_\alpha(p)$ which is defined
as  the smallest number such that every non-zero residue  class modulo $p$ can be represented by a   $p^\alpha$-smooth square-free integer 
not exceeding  $M_\alpha(p)$.  Furthermore, in~\cite{MSY} one can find new lower bounds on the number of $y$-smooth  square-free integers  $n \le x$ in a given residue class modulo $p$.

Here we obtain function field analogues of the bounds~\eqref{eq;Energy Sqrt} 
and~\eqref{eq;Energy Inv}, use these to estimate some bilinear sums
and then give applications to the distribution of smooth square-free polynomials 
in residue classes modulo an irreducible polynomial $F$ over a finite field.

\subsection{New set-up}\label{sec:setup}  Here we consider the above questions in the setting of function fields over finite fields.

We fix an odd prime power $q$ and an irreducible polynomial $F(X)$ over $\F_q$ of degree $r$, and consider the finite field $\F_q[X]/F(X)$. One can notice that, as fields, $\F_q[X]/F(X) \cong \F_{q^r}$.  For an integer $m$, we write $f \sim m$ to mean $f \in \F_{q^r}$, but $\deg f < m$ when we view $f \in \F_q[X]/F(X)$. 
In particular, if $\rho$ is a root of $F$, then we identify the following two sets
\begin{align*} 
\{f( X) & \in \F_q[X]/F(X):~f \sim m \} \\
& = \{u_0+u_1\rho + \ldots + u_{m-1} \rho^{m-1}:~  u_0, u_1, \ldots, u_{m-1} \in \F_q\}. 
\end{align*}  
Thus below we switch freely between the languages of function fields $\F_q[X]/F(X)$ and of finite fields $\F_{q^r}$.

As in previous works~\cite{DKSZ,KSSZ,SSZ1} we attach to our variables some complex weights. 
Namely, given two positive integers $m,n \leq r$ and two sequences of complex weights
\begin{equation} \label{Weights}
 \balpha = (\alpha_f)_{f \sim m }   \mand \bbeta = (\beta_g)_{g \sim n} , 
\end{equation} 
we denote 
$$\|\balpha\|_\infty = \max_{f \sim m }|\alpha_f| \mand  \|\balpha\|_\sigma = \(\sum_{f \sim m }|\alpha_f|^\sigma\)^{1/\sigma} \quad (\sigma\ge 1), $$
and  similarly for $\bbeta$. It is also convenient to define the weight $\bm{1}_m$ to simply be the characteristic function for $f \sim m$. 

We recall, the notation $U = O(V)$, 
$U \ll V$ and $ V\gg U$  are equivalent to $|U|\leqslant c V$ for some positive constant $c$, 
which throughout the paper may depend on the size $q$ of the ground field.

\begin{rem} We note that the condition that $q$ is odd is only needed for 
results concerning square roots. Other statements, such as 
Theorems~\ref{thm:AddEnergy-inv}, \ref{thm:BilinearBound-inv}, 
\ref{Thm:MBound} and~\ref{thm:Psi-LB} hold for any $q$. 
\end{rem}  

\section{Statements of the results}


\subsection{Energy of square roots and reciprocals}

For a weight $\bbeta$ as in~\eqref{Weights} we now define the weighted additive energy 
 $$ E_{q,r}^{\sqrt}(\bbeta) = \sum_{\substack{(u,v,x,y) \in \F_{q^r}^4 \\ u+v = x+y}}\beta_{u^2}\ov{\beta}_{v^2}\beta_{x^2}\ov{\beta}_{y^2}.$$

Note it is not difficult to see that $E_{q,r}^{\sqrt}(\bbeta)$ is a non-negative real number, see~\eqref{eq:E=Q^2} below. It is also important to recognize the special case
\begin{align*} 
E_{q,r}^{\sqrt}(\bm{1}_m)=  \#\{(u,v,x,y) \in \F_{q^r}^4: ~  u + v &= x +y , \\
 & \qquad   u^2,  v^2,  x^2,   y^2  \sim  m\}. 
\end{align*}  
In particular our goal is to improve the trivial bound 
\begin{equation} \label{eq:TrivBound-E}
E_{q,r}^{\sqrt}(\bm{1}_m) \ll q^{3m}.
\end{equation}

Here we establish the following:

 \begin{thm}\label{thm:AddEnergy-sqrt}
For any positive integer $m \leq r$ and for a weight $\bbeta$ as in~\eqref{Weights}  we have 
$$
E_{q,r}^{\sqrt}(\bbeta) \le
    |\bbeta\|_1^2|\bbeta\|_\infty^2 q^{m/2+o(m)} \(q^{m-r/2}+1\).
    $$
\end{thm}  

In particular, by Theorem~\ref{thm:AddEnergy-sqrt}  we have
$$
E_{q,r}^{\sqrt}(\bm{1}_m) \ll q^{7m/2-r/2}  + q^{5m/2}
$$
which is always stronger than~\eqref{eq:TrivBound-E} (unless $m$ is very close to
$r$ when no nontrivial bound is  possible). 

To formulate our next result we define
\begin{align*} 
E_{q,r}^{\inv}(m)=  \#\{(u,v,x,y) \in \F_{q^r}^4: ~  u^{-1} + v^{-1} &= x^{-1} +y^{-1} , \\
 & \qquad   u,  v,  x,   y  \sim m\}. 
\end{align*}  
We then have the following result:

 \begin{thm}\label{thm:AddEnergy-inv}
For any positive integer $m \leq r$ we have
$$
E_{q,r}^{\inv}(m) \le \(q^{(7m-r)/2}+ q^{2m} \) q^{o(m)}.
$$
\end{thm}

As before, we see that Theorem~\ref{thm:AddEnergy-inv} always 
improves the trivial bound $E_{q,r}^{\inv}(m)  \ll q^{3m}$.


\subsection{Bilinear sums with square roots and reciprocals}
We start with bounds on the sum 
$$W_{q,r}^{\sqrt}(\balpha, \bbeta; m,n) = \sum_{f \sim m}\sum_{g \sim n}\alpha_f \beta_g \sum_{\substack{h\in \F_{q^r}\\h^2=fg}}\psi(h) $$
where $\psi$ is a fixed nontrivial additive character of $\F_{q^r}$. 

One would naturally look to improve upon the trivial bound
\begin{equation} \label{eq:TrivBound-W}
W_{q,r}^{\sqrt}(\balpha, \bbeta; m,n) = O(\|\balpha\|_1|\bbeta\|_1). 
\end{equation} 

The following result does so, in certain ranges of $m$ and $n$.

\begin{thm}\label{thm:BilinearBound-sqrt}
For any positive integers $m,n \leq r$ and any weights as in~\eqref{Weights} we have
\begin{align*}
    |W_{q,r}^{\sqrt}(\balpha, \bbeta; m,n)|
     \leq \|\balpha\|_2|\bbeta\|_1^{3/4}|\bbeta\|_\infty^{1/4}&q^{r/8 + 5m/16 + n/16 + o(r)}\\
    &\(q^{m/8-r/16} + 1\)\(q^{n/8-r/16} + 1\). 
\end{align*}
\end{thm}

Clearly in the most interesting range $m,n \le r/2$ and for the weights $|\alpha_f|, |\beta_g| \le 1$ the bound 
of Theorem~\ref{thm:BilinearBound-sqrt} becomes 
$$
|W_{q,r}^{\sqrt}(\balpha, \bbeta; m,n)|
     \leq  q^{r/8 + 13m/16 + 13n/16  + o(r)}, 
$$
which is better than the trivial bound~\eqref{eq:TrivBound-W} provided that 
$$
m+n \ge 2r/3.
$$

Similarly, let 
$$W_{q,r}^{\inv}(\balpha, \bbeta; m,n) = \sum_{f \sim m}\sum_{g \sim n}\alpha_f \beta_g 
\psi\(f^{-1}g^{-1}\), $$
where as before $\psi$ is a fixed nontrivial additive character of $\F_{q^r}$. 

We prove the following:

\begin{thm}\label{thm:BilinearBound-inv}
For any positive integers $m, n\leq r$ and any weights as in~\eqref{Weights} we have
\begin{align*}
 W_{q,r}^{\inv}(\balpha, \bbeta; m,n)  
    \leq \|\balpha\|_\infty \|\bbeta\|_\infty & q^{r/8+3m/4 + 3n/4 + o(r)}\\
    & \(q^{3m/16-r/16} +1\)\( q^{3m/16-r/16}+1\). 
\end{align*}
\end{thm}

For $m,n \le r/3$ and for the weights $|\alpha_f|, |\beta_g| \le 1$ the bound 
of Theorem~\ref{thm:BilinearBound-inv} becomes 
$$
|W_{q,r}^{\sqrt}(\balpha, \bbeta; m,n)|
     \leq  q^{r/8+3m/4 + 3n/4 + o(r)},
$$
which is better than an analogue of~\eqref{eq:TrivBound-W} provided that 
$$
m+n \ge r/2.
$$

\subsection{Application to special polynomials in residue classes}
Directly analogous to the definition for integers, for any positive real number $k$,  we call a polynomial $f(X) \in \F_{q}[X]$ {\it $k$-smooth} if $f$ has no irreducible factors of degree exceeding $k$. 

Recalling that $F$ is some irreducible polynomial of degree $r$ over $\F_q$, for any real $\alpha > 0$ we denote by $M_{\alpha, q}(F)$ the smallest integer such that any non-zero residue class in $\F_q(X)/F(X)$ contains an $\alpha r$-smooth square-free representative whose degree does not exceed $M_{\alpha, q}(F)$. We formally set $M_{\alpha, q}(F) = \infty$ if no such representative exists. To the authors' knowledge, it is not known exactly for which $F$ we have $M_{\alpha, q}(F) < \infty$, even for the case $\alpha = 1$.

\begin{thm}\label{Thm:MBound}
As $r \to \infty$, for any fixed $\alpha > 0 $ we have that for every monic, irreducible polynomial $F(X) \in \F_q[X]$ of degree $r$,
$$M_{\alpha, q}(F) \leq (2+o(1)) r. $$
\end{thm}

Now for any $a(X) \in \F_q[X]$ and positive integers $k$ and $m$ we define $\Psi(k,m;F, a)$ to be the number of $g(X) \in \F_q[X]$ satisfying
$$
g \equiv a \Mod{F}, \quad  \deg g < k, \ g \text{ is $m$-smooth}
$$
and similarly $\Psi^\#(k,m;F, a) $ to count those $g(X) \in \F_q[X]$  satisfying
\begin{equation} \label{eq:psisfcongruence}
     g \equiv a \Mod{F}, \quad  \deg g  < k,\ g \text{ is $m$-smooth and square-free}.
\end{equation}  
We remark that in the above we do not use the notation $g \sim k$ as it is defined
for polynomials in the residue ring $\F_q[X]/F(X)$ (thus makes sense 
only for $k \le r$), while here  $g(X) \in \F_q[X]$ and can be of degree much larger than $r$. 

We follow closely the ideas in~\cite{MSY} to derive the following lower bound on $\Psi^\#(k,m;F, a)$:

\begin{thm}\label{thm:Psi-LB}
For any fixed real numbers $\alpha$ and $\beta$ with $\beta \in (23/24, 1]$ and $\alpha \in (9/2-3\beta, 3\beta]$, and for any positive reals $k,m$  with $r\alpha \leq k \leq r(\alpha + o(1))$ and $r\beta \leq m \leq r(\beta + o(1))$ we have
$$\Psi^\#(k,m;F, a) \geq q^{k-r+o(r)}$$
for every monic, irreducible polynomial $F(X) \in \F_q[X]$ of degree $r$, as $r \to \infty$. 
\end{thm}

\section{Energy Bounds}
\subsection{Preparations} To prove Theorems~\ref{thm:AddEnergy-sqrt} and~\ref{thm:AddEnergy-inv}, we
need the following two results given in~\cite{CillShp}. 

\begin{lem}\label{lem:f(x)=y}
Let $P$ be a polynomial of degree $2$ over $\F_{q^r}$. For any positive integer $m \leq r$, the number of solutions to the equation 
$$P(u)  = v , \qquad u, v \sim m,$$
is bounded by 
$\(q^{-m/2} + q^{-(r-m)/2}\)q^{m+o(m)}$.
\end{lem}

\begin{lem}\label{lem:Hyperb}
 For any positive  integer $ m \leq r$ and any $a \in \F_{q^r}^*$, the number of solutions to the equation 
$$uv  = a , \qquad u, v \sim m, $$
is bounded by 
$\(q^{(3m-r)/2}+1\)q^{o(m)}$.
\end{lem}

\subsection{Proof of Theorem~\ref{thm:AddEnergy-sqrt}}
For any $\lambda \in \F_{q^r}$ we define 
$$Q_{\lambda}(\bbeta) = \sum_{\substack{(u,v) \in \F_{q^r}^2 \\ u-v = \lambda}}\beta_{u^2}\ov{\beta}_{v^2}. $$
We note that together with each  term $\beta_{u^2}\ov{\beta}_{v^2}$ corresponding to $u-v = \lambda$, 
the above sum also contains the term 
$$\beta_{(-v)^2}\ov{\beta}_{(-u)^2} = \ov \beta_{u^2} \beta_{v^2}, 
$$
corresponding to $(-v) - (-u)= \lambda$. Hence $Q_{\lambda}(\bbeta)$ is real. 

Subsequently, we observe that 
\begin{equation} 
 \label{eq:E=Q^2}
E_{q,r}^{\sqrt}(\bbeta)  = \sum_{\lambda \in \F_q}Q^2_{\lambda}(\bbeta) = \sum_{\lambda \in \F_q}|Q_{\lambda}(\bbeta)|^2. 
\end{equation}

Note that by the triangle inequality we have 
\begin{equation} \label{Qlambda_triangleinequality}
\begin{split} 
    \sum_{\lambda \in \F_{q^r}}|Q_{\lambda}(\bbeta)|
    &\leq \sum_{\lambda \in \F_{q^r}} \sum_{\substack{u,v\in \F_{q^r} \\ u-v = \lambda}}|\beta_{u^2}||{\beta_{v^2}}{|}   \\
    &=  \sum_{\substack{u,v \in \F_{q^r}}}|\beta_{u^2}||{\beta_{v^2}}{|}   
     \leq 4\sum_{\substack{x,y \in \F_{q^r}}}|\beta_{x}||{\beta_{y}}{|}   
     \ll \|\bbeta\|_1^2, 
\end{split} 
\end{equation} 
which is  used later. Now, we have  
$$
     E_{q,r}^{\sqrt}(\bbeta)
    = \sum_{\lambda \in \F_{q^r}^*}Q^2_{\lambda}(\bbeta) + Q_0^2(\bbeta)
    =  \sum_{\lambda \in \F_{q^r}^*}|Q_{\lambda}(\bbeta)|^2 + O\(\|\bbeta\|_2^4 \), 
$$
which gives 
\begin{equation} \label{eq:E_{q,r}-sqrt}
    | E_{q,r}^{\sqrt}(\bbeta)|
\leq \max_{\lambda \in \F_{q^r}^*}|Q_{\lambda}(\bbeta)|\sum_{\lambda \in \F_{q^r}^*}|Q_{\lambda}(\bbeta)| + O\(\|\bbeta\|_2^4 \).
\end{equation} 
To now deal with the term $ \max_{\lambda \in \F_{q^r}^*}|Q_{\lambda}(\bbeta)|$ we notice
\begin{equation} \label{maxQlambda}
\begin{split} 
    \max_{\lambda \in \F_{q^r}^*}|Q_{\lambda}(\bbeta)|
    &= \max_{\lambda \in \F_{q^r}^*}\left|\sum_{g_1 \sim m}\sum_{g_2 \sim m}\sum_{\substack{u,v \in \F_{q^r} \\ u-v = \lambda\\ u^2 = g_1, \ v^2 = g_2}}\beta_{u^2}\ov{\beta}_{v^2}\right|   \\
    &\leq \max_{\lambda \in \F_{q^r}^*}\sum_{g_1 \sim m}\sum_{g_2 \sim m}\sum_{\substack{u,v \in \F_{q^r} \\ u-v = \lambda\\ u^2 = g_1, \ v^2 = g_2}}|\beta_{u^2}\ov{\beta}_{v^2}|   \\
    &\leq \max_{g \sim m}|\beta_g|^2\max_{\lambda \in \F_{q^r}^*}\sum_{g_1 \sim m}\sum_{g_2 \sim m}\sum_{\substack{u,v \in \F_{q^r} \\ u-v = \lambda\\ u^2 = g_1, \ v^2 = g_2}}1  \\
    &= |\bbeta\|_\infty^2Q_{\lambda}(\bm{1}_m), 
    \end{split} 
\end{equation} 
 where  we recall that $\bm{1}_m$ denotes  the characteristic function on $g \in \F_{q^r}$ for $g \sim m$.
 
We next show 
$$
    Q_{\lambda}(\bm{1}_m) \leq 4\#\{(Z,V) \in \F_{q^r}^2:~(Z-\lambda^2)^2 = 4\lambda^2V, \: Z \sim m, \: V \sim m\}.
$$
To see this, we firstly have
$$
    Q_{\lambda}(\bm{1}_m) = \#\{(u,v) \in \F_{q^r}^2 :~ u-v = \lambda, u^2 \sim m, \: v^2 \sim m\}.
$$
If we set $U = u^2$ and $V = v^2$ then using $u-v = \lambda$ we see 
$$U-V = u^2 - v^2 = (u-v)(u+v) = \lambda(2v + \lambda). $$
Rearranging and squaring, we obtain
$$(U-V-\lambda^2)^2 = 4\lambda^2V $$
and letting $Z = U-V$ we have 
$$(Z-\lambda^2)^2 = 4\lambda^2V. $$
Given any $(Z,V)$ satisfying the above equation, this corresponds to at most $4$ pairs $(u,v)$. Thus we can say 
$$Q_{\lambda}(\bm{1}_m) \leq 4\#\{(Z,V) \in \F_{q^r}^2 :~ (Z-\lambda^2)^2 = 4\lambda^2V, \: Z \sim m, \: V \sim m\}$$ 
as desired. Now, using Lemma~\ref{lem:f(x)=y} we obtain 
\begin{align*}
Q_{\lambda}(\bm{1}_m) &   \leq q^{m+ o(m)}(q^{-m/2}  + q^{-(r-m)/2})\\
& =  \(q^{m/2}  + q^{3m/2-r/2}\)q^{o(m)}.
\end{align*}

Substituting this into~\eqref{maxQlambda} we obtain
$$
    \max_{\lambda \in \F_{q^r}^*}|Q_{\lambda}(\bbeta)|
     \ll |\bbeta\|_\infty^2 \(q^{m/2}  + q^{3m/2-r/2}\)q^{o(m)}. 
$$
We can in turn substitute this into~\eqref{eq:E_{q,r}-sqrt}, and also use~\eqref{Qlambda_triangleinequality}, to derive
$$
E_{q,r}^{\sqrt}(\bbeta)\ll 
    |\bbeta\|_1^2|\bbeta\|_\infty^2 \(q^{m/2}  + q^{3m/2-r/2}\)q^{o(m)} + |\bbeta\|_2^4 
$$ 
and since 
$$|\bbeta\|_2 \leq |\bbeta\|_1^{1/2}|\bbeta\|_\infty^{1/2} $$
we arrive at 
$$
 E_{q,r}^{\sqrt}(\bbeta) \le
    |\bbeta\|_1^2|\bbeta\|_\infty^2 \(q^{m/2}  + q^{3m/2-r/2}\)q^{o(m)}, 
    $$
which concludes the proof.  

\subsection{Proof of Theorem~\ref{thm:AddEnergy-inv}}
We denote by $I_{F}(a, m)$ the number of solutions to the equation 
$$u^{-1} + v ^{-1} = a, \qquad u, v \sim m.
$$
Rearranging we have
$$(u-a^{-1})(v-a^{-1}) = a^{-2}. $$
Now, applying  Lemma~\ref{lem:Hyperb}, for $a\ne 0$, we  derive
$$I_{F}(a, m)\leq \(q^{(3m-r)/2}+1\)q^{o(m)}.$$

We also have the trivial bound $I_{F}(0, m) \le q^m$. 
Thus we can write 
\begin{align*}
E_{q,r}^{\inv}(m) 
    &= \sum_{a \sim r}I_{F}(a, m)^2 \\
    &\leq q^{2m} + \sum_{a \sim r, \, a \neq 0}I_{F}(a, m)^2 \\
    &\leq  q^{2m} + \(q^{(3m-r)/2}+1\) q^{o(m)}\sum_{a \sim r}I_{F}(a, m) \\
    &\leq q^{2m + o(m)} \(q^{(3m-r)/2}+1\), 
\end{align*}
which gives the desired result.

\section{Bounds of Bilinear Sums}
\subsection{Preparations} Before proving Theorem~\ref{thm:BilinearBound-inv}, we need the following result, which is analogous to~\cite[Chapter~6, Exercise~14]{V1955}. 

\begin{lem}\label{lem:Vinogradovbound}
Let $\psi$ be a nontrivial additive character of $F_{q^r}$. For any complex weights
as in~\eqref{Weights} (with $m=n=r$) we have  
$$\left|\sum_{f \sim r}\sum_{g \sim r}\alpha_f\beta_g\psi(fg)\right| \leq  q^{r/2}\|\balpha\|_2\|\bbeta\|_2.$$

\end{lem}
\begin{proof}
Let
$$
S = \sum_{f \sim r}\sum_{g \sim r}\alpha_f\beta_g\psi(fg).
$$ 
Applying the Cauchy--Schwarz inequality and changing the order of summation we have
\begin{align*}
    |S|^2
    &\leq  \sum_{f \sim r}|\alpha_f|^2 \sum_{f \sim r}\left|\sum_{g \sim r}\beta_g\psi(fg) \right|^2 \\
    &= \|\balpha\|_2^2\sum_{g_1, g_2 \sim r}\beta_{g_1}\ov{\beta}_{g_2}\sum_{f \sim r}\psi(f(g_1-g_2)).
\end{align*}
Now for a given pair $(g_1, g_2)$, the inner sum vanishes unless $g_1 = g_2$ in which case it is equal to $q^r$. So we have 
$$
    |S|^2 \leq \|\balpha\|_2^2\sum_{g \sim r}|\beta_g|^2q^r = q^r\|\balpha\|_2^2\|\bbeta\|_2^2,
$$
as desired. 
\end{proof}

\subsection{Proof of Theorem~\ref{thm:BilinearBound-sqrt}}
Recall
$$W_{q,r}^{\sqrt}(\balpha, \bbeta; m,n) = \sum_{f \sim m}\sum_{g \sim n}\alpha_f \beta_g \sum_{\substack{h\in \F_{q^r}\\h^2=fg}}\psi(h). $$
In the following expansion we first apply the Cauchy--Schwarz inequality, then expand the squared term and then rearrange the order of summation, which yields
\begin{align*}
    |W_{q,r}^{\sqrt}(\balpha, \bbeta; m,n)|^2
    &\leq \sum_{f \sim m}|\alpha_f|^2  \sum_{f \sim m}\left|\sum_{g \sim n}\sum_{\substack{h\\h^2=fg}}\beta_g\psi(h) \right|^2 \\
    &= \|\balpha\|_2^2\sum_{\substack{g_1 \sim n\\g_2 \sim n}}{\beta_{g_1}}\ov{\beta}_{g_2}\sum_{f \sim m}\sum_{\substack{u,v\\u^2 = fg_1\\v^2 = fg_2}}\psi(u-v). 
\end{align*}
Now we   write
\begin{equation} \label{R1+R-1}
    |W_{q,r}^{\sqrt}(\balpha, \bbeta; m,n)|^2= \|\balpha\|_2^2\(R_1 + R_{-1}\), 
\end{equation} 
where 
$$
    R_j 
    =
    \sum_{\substack{g_1 \sim n\\g_2 \sim n \\ \chi(g_1) = \chi(g_2) = j}}{\beta_{g_1}}\ov{\beta}_{g_2}\sum_{\substack{f \sim m\\\chi(f) = j}}\sum_{\substack{u,v\\u^2 = fg_1\\v^2 = fg_2}}\psi(u-v)
$$
and $\chi$ is the quadratic character of $\F_{q^r}$ (note that since $q$ is odd, such a quadratic character exists).

It suffices to only consider $R_1$ since $R_{-1}$ can be worked through identically (see~\cite{SSZ1}). Now to simplify $R_1$ we can write 
$$
    R_1 
=   \frac{1}{2}\sum_{\substack{f \sim m}}\sum_{\substack{t \\t^2 = f}}\sum_{\substack{g_1 \sim n\\g_2 \sim n }}\sum_{\substack{u,v\\u^2 = g_1\\v^2 = g_2}}{\beta_{u^2}}\ov{\beta}_{v^2}\psi(ut-vt), 
$$
and now collecting the terms with the same value of $u-v$ we have 
$$
    R_1 
      =   \frac{1}{2}\sum_{\substack{f \sim m}}\sum_{\substack{t\in \F_{q^r} \\t^2 = f}}\sum_{\lambda \in \F_{q^r}}\sum_{\substack{u,v \in \F_{q^r}\\u-v = \lambda}}{\beta_{u^2}}\ov{\beta}_{v^2}\psi(t\lambda).
$$
In our sum, we are setting $\beta_x = 0 $ if $x\not\sim n$. 
We can now write $R_1$ as 
\begin{equation} \label{R1_simplified}
    R_1 =   \frac{1}{2}\sum_{\lambda \in \F_q^r}A_{\lambda}Q_{\lambda}(\bbeta), 
\end{equation} 
where 
$$A_{\lambda} = \sum_{f \sim m}\sum_{\substack{t \in \F_{q^r}\\ t^2 = f}}\psi(t\lambda ).$$
We next show 
\begin{equation} \label{Alambda_bound}
    \sum_{\lambda \in \F_{q^r}}|A_\lambda|^4 = q^r E_{q,r}^{\sqrt}\(\bm{1}_m\), 
\end{equation} 
where, as before,  $\bm{1}_m$ denotes  the characteristic function on $f \in \F_{q^r}$ with $f \sim m$. Expanding out the left we have 
\begin{align*}
    \sum_{\lambda \in \F_{q^r}}|A_\lambda|^4
    &= \sum_{\lambda \in \F_{q^r}}\left|\sum_{f \sim m}\sum_{\substack{t \in \F_{q^r}\\ t^2 = f}}\psi(t\lambda ) \right|^4\\
    &= \sum_{\substack{f_1, f_2, \\f_3, f_4 \sim m}}\, \sum_{\substack{t_1, t_2, t_3, t_4 \\ t_1^2 = f_1 \ t_2^2 = f_2 \\ t_3^2 = f_3 \ t_4^2 = f_4}}\sum_{\lambda \in \F_{q^r}}\psi(\lambda(t_1 + t_2 - t_3 - t_4)).
\end{align*}
By orthogonality the inner most sum vanishes unless $t_1 + t_2 = t_3 + t_4 $, in which case it is equal to $q^r$. Thus we have 
\begin{align*}
    \sum_{\lambda \in \F_{q^r}}|A_\lambda|^4
    &= q^r E_{q,r}^{\sqrt}\(\bm{1}_m\).
\end{align*}
Now we can trivially write 
$$
    |Q_{\lambda}(\bbeta)| = \(|Q_{\lambda}(\bbeta)|^2\)^{1/4}|Q_{\lambda}(\bbeta)|^{1/2}
$$
so from~\eqref{Qlambda_triangleinequality},  \eqref{R1_simplified},  \eqref{Alambda_bound} and the H\"{o}lder  inequality we have
\begin{equation} \label{eq:R4E}
\begin{split} 
    |R_1|^4 
  &\le \(\sum_{\lambda \in \F_q^r}|A_{\lambda}||Q_{\lambda}(\bbeta)|^{1/2}\(|Q_{\lambda}(\bbeta)|^2\)^{1/4}\)^4\\
     &\le \sum_{\lambda \in \F_{q^r}}|A_\lambda|^4\sum_{\lambda \in \F_{q^r}}|Q_\lambda(\bbeta)|^2\(\sum_{\lambda \in \F_q^r}|Q_{\lambda}(\bbeta)|\)^2\\
     &\le q^r|\bbeta\|_1^4 E_{q,r}^{\sqrt}\(\bm{1}_m\) E_{q,r}^{\sqrt}(\bbeta).
\end{split} 
\end{equation}

Now, using Theorem~\ref{thm:AddEnergy-sqrt} we obtain 
$$ |E_{q,r}^{\sqrt}(\bbeta)| \le |\bbeta\|_1^2|\bbeta\|_\infty^2\(1 + q^{n-r/2}\)  q^{n/2 + o(n)} 
$$
and 
\begin{align*}
     E_{q,r}^{\sqrt}\(\bm{1}_m\) 
    &\le  ||\bm{1}_m||_1^2||\bm{1}_m||_\infty^2 \(1 + q^{m-r/2}\)  q^{m/2 + o(m)}  \\
    &\le q^{5m/2+ o(m)}\(1 + q^{m-r/2}\), 
\end{align*}
so together with~\eqref{eq:R4E} this gives
$$
   |R_1|^4
 \le |\bbeta\|_1^6|\bbeta\|_\infty^2q^{r + 5m/2 + n/2 + o(r))}\(1 + q^{m-r/2}\)\(1 + q^{n-r/2}\).
$$
Finally using~\eqref{R1+R-1} we conclude 
\begin{align*}
    |W_{q,r}^{\sqrt}(\balpha, \bbeta; m,n)|
   \leq \|\balpha\|_2|\bbeta\|_1^{3/4}|\bbeta\|_\infty^{1/4}&q^{r/8 + 5m/16 + n/16 + o(r)}\\
    & \(q^{m/8-r/16} + 1\)\(q^{n/8-r/16} + 1\),
\end{align*}
and the result follows.  

\subsection{Proof of Theorem~\ref{thm:BilinearBound-inv}}
Recall 
$$W_{q,r}^{\inv}(\balpha, \bbeta; m,n) = \sum_{f \sim m}\sum_{g \sim n}\alpha_f \beta_g 
\psi\(f^{-1}g^{-1}\). $$
Applying the Cauchy--Schwarz inequality to the sum over $f$ and rearranging, we obtain \begin{align*}
    &|W_{q,r}^{\inv}(\balpha, \bbeta; m,n)|^2\\
    &\hspace{3em}\leq \|\balpha\|^2_\infty\|\bbeta\|^2_\infty q^n\sum_{f\sim n}\left|\sum_{g \sim m}\psi(af^{-1}g^{-1}) \right|^2\\
    &\hspace{3em}\leq \|\balpha\|^2_\infty\|\bbeta\|^2_\infty q^n\sum_{g_1 \sim n}\sum_{g_2 \sim n}\left|\sum_{f \sim m}\psi(f^{-1}(g_1^{-1} - g_2^{-1}))\right|.
\end{align*}
Now  applying the Cauchy-Schwarz inequality to the sums over $g_1, g_2$ we derive
\begin{align*}
    &|W_{q,r}^{\inv}(\balpha, \bbeta; m,n)|^4\\
    &\hspace{1em}\leq 
     \|\balpha\|^4_\infty\|\bbeta\|^4_\infty q^{2n + 2m}
     \sum_{f_1, f_2 \sim m} \, \sum_{g_1,g_2 \sim n} 
     \psi((f_1^{-1} - f_2^{-1})(g_1^{-1}-g_2^{-1}))\\
    &\hspace{1em} =  \|\balpha\|^4_\infty\|\bbeta\|^4_\infty q^{2n+2m}\sum_{u \sim r}\sum_{v \sim r}I_{F}(u, m)I_{F}(v, n)\psi(u v),
\end{align*} 
where $I_{F}(u, m)$ and $I_{F}(v, n)$ are as defined in the proof of Theorem~\ref{thm:AddEnergy-inv}. 

Now applying Lemma~\ref{lem:Vinogradovbound} we have 
\begin{align*}
    &|W_{q,r}^{\inv}(\balpha, \bbeta; m,n)|^8\\
    &\hspace{3em}\leq 
    \|\balpha\|^8_\infty\|\bbeta\|^8_\infty q^{4n+4m}q^r
    \(\sum_{u \sim r}I_{F}(u, m)^2  \) \(\sum_{v \sim r}I_{F}(v, n)^2 \).
\end{align*}
Finally applying Theorem~\ref{thm:AddEnergy-inv} to these sums we obtain
\begin{align*}
    &|W_{q,r}^{\inv}(\balpha, \bbeta; m,n)|^8\\
    &\qquad \qquad \leq \|\balpha\|^8_\infty\|\bbeta\|^8_\infty q^{r+6n + 6m + o(r)}
    \(q^{(3n-r)/2}+1\)\(q^{(3m-r)/2}+1\), 
\end{align*}
and the result follows.

\section{Results of polynomials in residue classes}

\subsection{Preparations} For convenience, for any positive integer $n < r$ we introduce the quantity
$$
    B(r,n) =
    \begin{cases}
    {3n}/{2} + {r}/{8}, &n <{r}/{3}, \\
    {15n}/{8}, &{r}/{3} \leq n < r.
    \end{cases}
$$
We also denote by $\mathcal{P}_n$ the set of all monic, irreducible polynomials of degree exactly $n$ in $\F_q[X]$. 
 Note that we can naturally identify $\cP_n$ with a subset of $\F_{q^r}$, via the discussion in Section \ref{sec:setup}.

The following is a direct corollary of Theorem~\ref{thm:BilinearBound-inv}.
\begin{lem}\label{lem:sumreciprocalirreducibles}
For any positive integer $n< r$ and any nontrivial additive character $\psi$ of $\F_{q^r}$ we have 
$$
    \left|\sum_{\ell_1,\ell_2 \in \mathcal{P}_{n}}\psi(\ell_1^{-1}\ell_2^{-1}) \right| \leq q^{B(r,n) + o(r)}.
$$
\end{lem}

Now we introduce a number of results regarding bounds on the number of solutions to certain equations over $\F_{q^r}$. For any positive integers $n < r$, $h \leq r$ and any $a \in \F_{q^r}^*$ we denote by $N_F(a,n,h)$ the number of solutions to 
\begin{equation} \label{eq:Ncongruence}
    \ell_1\ell_2u = a, \quad \ell_1,\ell_2 \in \cP_{n}, \: u \sim h. 
\end{equation}

The next two results give two different bounds for $N_F(a,n,h)$.

Let
$$
\varpi_{n} = \#\cP_{n}.
$$
In particular, we have
\begin{equation} \label{eq:CountIrred}
\varpi_{n}=\frac{1}{n} \sum_{d \mid n} \mu(d) q^{n/d} = \frac{1}{n}\(q^n + O\(q^{n/2}\)\), 
\end{equation} 
where $\mu$ is the classical M{\"o}bius function, see~\cite[Theorem~3.25]{LN}. 

\begin{lem}\label{lem:Nbound}
For any positive integers $n < r$, $h \leq r$ and any $a \in \F_{q^r}^*$ we have 
$$N_{F}(a, n, h)= \varpi_{n}^2q^{h-r} + O\(q^{B(r,n) + o(r)}\).$$

\end{lem}

\begin{proof}
Recalling our discussion in Section \ref{sec:setup}, if $\rho$ is a root of $F(X)$ we identify $\F_{q^r}$ as the $\F_q$-span of $\{1, \rho, ..., \rho^{r-1}\}$. Now, let ${\omega} = \{\omega_0, ..., \omega_{r-1}\}$ be the basis dual to $\{1, \rho, ..., \rho^{r-1}\}$ in $\F_{q^r}$. That is, ${\omega}$ satisfies $$\text{Tr}_{\F_{q^r}/\F_q}(\rho^i\omega_j) = \delta_{i,j}. $$
Now we have the following orthogonality relation;
$$
    \prod_{j=h}^{r-1}\sum_{b \in \F_q}\eta(b\text{Tr}_{\F_{q^r}/\F_q}(u\omega_j))
     = 
    \begin{cases}
    q^{r-h}, &u \sim h, \\
    0, &\text{otherwise,}
    \end{cases}
$$
where $\eta$ is an additive character of $\F_q$. In the case of $u \sim h$, rearranging we obtain 
\begin{align*}
    q^{r-h}
    &= \sum_{(b_{h}, ..., b_{r-1}) \in \F_{q^r}^{r-h}}\eta\(\text{Tr}_{\F_{q^r}/\F_q}\sum_{j=h}^{r-1}b_j(u\omega_j) \)\\
    &= \sum_{b \in \mathcal{H}}\psi(bu),
\end{align*}
where $\psi$ is a lift of $\eta$ to $\F_{q^r}$  and 
$$
    \mathcal{H} =\left\{ \sum_{j=h}^{r-1}b_j\omega_j :~ b_j \in \F_q \right\}.
$$
Thus, we can write 
$$
    N_{F}(a, n, h)= 
    \sum_{f,g \in \mathcal{P}_{n}}\frac{1}{q^{r-h}}\sum_{b \in \mathcal{H}}\psi(baf^{-1}g^{-1}).
$$
We now rearrange, separate the contribution from $b=0$ and apply Lemma~\ref{lem:sumreciprocalirreducibles} to get
\begin{align*}
     N_{F}(a,n,h)
    &= 
    \frac{1}{q^{r-h}}\sum_{b \in \mathcal{H}}\sum_{f,g \in \mathcal{P}_{L}}\psi(baf^{-1}g^{-1})\\
    &= \varpi_{n-1}^2q^{-r+h} +  \frac{1}{q^{r-h}}\sum_{b \in \mathcal{H}\setminus \{0\}}\sum_{f,g \in \mathcal{P}_{n}}\psi(baf^{-1}g^{-1})\\
    &= \varpi_{n}^2q^{-r+h} + O\(q^{B(r,n) + o(r)}\),
\end{align*}
which concludes the proof.
\end{proof}

\begin{lem}\label{lem:simpleNbound}
For any positive integers $n < r$, $h \leq r$ and any $a \in \F_{q^r}^*$ we have 
$$N_{F}(a, n, h) \leq q^{o(r)}\(1 + q^{2n + h -r}\).$$
\end{lem}
\begin{proof}
We discuss this in the language of polynomials, and of course we can view $a$ as a polynomial over $\F_q$ with $\deg a < r$. Thus, the congruence~(\ref{eq:Ncongruence}) implies $\ell_1\ell_2u = a + kF$ for some $k \in \F_q[X]$ such that $\deg k  \leq 2n + h -r -1$. Thus, $k$ takes on at most $q^{2n+h-r +1} + 1 $ values. So by~\cite[Theorem~1]{CillShp}, for each such $k$ we have that $\ell_1, \ell_2$ can each take on at most $q^{o(r)}$ values among the divisors of $a + kF$, after which of course $u$ is uniquely determined. 
\end{proof}

Now , let $N^\#_{F}(a, n, h)$ denote the number of solutions to the congruence~(\ref{eq:Ncongruence})  with square-free $u$. 
Then we have the following (where as before $\varpi_{n} = \#\mathcal{P}_{n}$). 

\begin{lem}\label{lem:Nboundsqfree}
For any $a \in \F_{q^r}^*$ and any positive integers $n,h$ with $n < r$, $h \leq r$ and a non-negative integer $ d \leq h/2$ we have 
\begin{align*}
&N_{F}^\#(a, n, h)\\
&\qquad \quad = \varpi_{n}^2q^{h-r}\(\frac{q-1}{q^2}\)  + O\(\(q^{2n+h-d-r} + q^{B(r,L) + d} + q^{h/2}\)q^{o(r) } \).
\end{align*}
\end{lem}
\begin{proof} It is  convenient to introduce an analogue of  the classical M{\"o}bius function 
$\mu$ for polynomials  over $\F_q[X]$:
\begin{equation} 
\label{eq:PolyMobius}
\begin{split} 
\mu_q(g) 
=
\begin{cases}
(-1)^k, &g \text{ is square-free and a product of $k$ distinct}\\& \text{irreducible factors,}\\
0, &\text{otherwise}.
\end{cases}
\end{split} 
\end{equation}  
By inclusion-exclusion we have
$$
    N_{F}^\#(a, n, h)= \sum_{g \sim \fl{h/2} + 1}\mu_q(g)N_{F}\(ag^{-2},n,h-2\deg g \).
$$
Now, firstly considering when $d \leq \deg g \leq \fl{h/2}$ we have by 
Lemma~\ref{lem:simpleNbound}
\begin{align*}
   &\sum_{d \leq \deg g \leq \fl{h/2}}N_{F}\(ag^{-2},n,h-2\deg g \)\\
  &\hspace{10em}\leq  \sum_{d \leq \deg g \leq \fl{h/2}} (1 + q^{2n + h -2\deg g -r})q^{o(r)}\\
   &\hspace{10em}\leq \(q^{h/2} + q^{2n + h - d - r}\) q^{o(r)}.
\end{align*}

Now considering $\deg g < d$, by Lemma~\ref{lem:Nbound} we have
\begin{align*}
    &\sum_{g \sim d}\mu_q(g)N_{F}(ag^{-2},n,h-2\deg g )\\
    &\hspace{1em}= \sum_{g \sim d}\mu_q(g)\(\varpi_{n}^2q^{h-2\deg g-r} + O\(q^{B(r,n) + o(r)}\)\)\\
    &\hspace{1em}= \varpi_{n}^2q^{h-r}\sum_{g \sim d}\frac{\mu_q(g)}{q^{2\deg g}} + O\(q^{B(r,n) + d + o(r) }\) \\
    &\hspace{1em} = \varpi_{n}^2q^{h-r}\sum_{g \in \F_q[X]}\frac{\mu_q(g)}{q^{2\deg g}} 
    + O\(\(\varpi_{n}^2q^{h-d-r} + q^{B(r,n) + d}\)q^{o(r)}\)\\
    &\hspace{1em} = \varpi_{n}^2q^{h-r}\(\frac{q-1}{q^2} \) + O\(\(\varpi_{n}^2q^{h-d-r} + q^{B(r,n) + d}\)q^{o(r)}\), 
\end{align*}
and the result follows. 
\end{proof}

Next, for any postive integers $n < r$, $h \leq r$ and any $a \in \F_{q^r}^*$ let $Q_{F}(a, n, h)$ count the number of solutions to 
\begin{equation} \label{eq:Qcongruence}
    \ell_1\ell_2^2v = a, \qquad \ell_1, \ell_2 \in \mathcal{P}_{n}, \: v \sim h.
\end{equation} 
Then we have the following:

\begin{lem}
For any positive integers $n < r$, $h \leq r$ with $n+h \leq r$ and any $a \in \F_{q^r}^*$ we have 
$$Q_{F}(a, n, h) \leq q^{n + o(r)}(q^{n+h-r} + 1). $$
\end{lem}
\begin{proof}
The congruence~\eqref{eq:Qcongruence} of course gives $\ell_1v = a\ell_2^{-2}$. Since $n+h \leq r$, for each choice of $\ell_2$ there are at most $q^{n + h -r} +1$ values for $\ell_1u$. The result then follows, recalling~\cite[Theorem~1]{CillShp}.
\end{proof}
For any positive integer $n$ we let $\cP_n$ denote the set of monic polynomials of degree exactly $n$ in $\F_q[X]$, and let $\cS_n \subset \cP_n$ denote those that are square-free. Again, recall that we can naturally identify each of these with a subset of $\F_{q^r}$ by the discussion in Section~\ref{sec:setup}.

Furthermore, let $\cX_F$ denote the set of multiplicative characters on the finite field $\F_q[X]/F(X) \cong \F_{q^r}$
and let  $\cX_F^*= \cX_F \setminus \{\chi_0\}$ be the set of non-principal characters.

To prove Theorem~\ref{Thm:MBound} we need the following result given in~\cite[Theorem~1.3]{Han} and~\cite[Theorem~1]{BLL}.

\begin{lem}\label{lem:MonicIrredSum}
For any $\chi \in \cX_F$ and positive integer $n < r$,
$$\left| \sum_{f \in \cP_{n}}\chi(f)\right| \leq q^{n/2 + o(r)}. $$
\end{lem}
This leads to the following:
\begin{lem}\label{lem:SqFreeSum}
For any $\chi \in  \cX_F$ and positive integer $n < r$,
$$
\left| \sum_{f \in \mathcal{S}_{n}}\chi(f)\right| 
    \leq 
    nq^{n/2  + o(r)}.
$$
\end{lem}

\begin{proof} Here we use definition~\eqref{eq:PolyMobius}
 of  the   M{\"o}bius function for polynomials.
By inclusion-exclusion we have
\begin{align*}
\left|\sum_{f \in \mathcal{S}_{n}}\chi(f)\right|
&= \left|\sum_{k \leq n}\sum_{d \in \cP_k}\mu_q(d)\sum_{\substack{f \in \cP_{n} \\ d^2 |f}}\chi(f)\right|\\
&= \left|\sum_{k \leq n}\sum_{d \in \cP_k}\mu_q(d)\sum_{\substack{gd^2 \in \cP_n}}\chi(gd^2)\right|\\
&= \left|\sum_{k \leq n/2}\sum_{d \in \cP_k}\mu_q(d)\chi(d^2)\sum_{\substack{g \in \cP_{n-2k}}}\chi(g)\right|\\
&\leq \sum_{k \leq n/2}\sum_{d \in \cP_k}\left|\sum_{g \in \cP_{n-2k}}\chi(g) \right|.
\end{align*}
Finally, by Lemma~\ref{lem:MonicIrredSum} we write
\begin{align*}
    \sum_{\substack{f \in \mathcal{S}_n}}\chi(f)
    &\leq \sum_{k \leq n/2}\sum_{d \in \cP_k}q^{(n-2k)/2 + o(r)} \\
    & =  \sum_{k \leq n/2}q^kq^{(n-2k)/2 + o(r)}  = q^{n/2  + o(r)}\sum_{k \leq n/2}1 \leq nq^{n/2 +o(r)}, 
\end{align*}
which concludes the proof. 
\end{proof}

\subsection{Proof of Theorem \ref{Thm:MBound}}
Fix some integer $n > \max\left\{2/\alpha, 2\right\}$ and a real number $\varepsilon > 0$. Next define $\beta = 1-2/n + \varepsilon$ and choose an integer $k > \max\left\{\rf{\beta /\alpha}, 2\right\}$.  We further denote 
$$
    T = \fl{\frac{2}{n}r}  \mand  W = \beta r
$$
and define the sets
\begin{itemize}
    \item $\mathcal{S} = \mathcal{S}_T$, as defined previously to be the set of square-free monic polynomials $s$ with $\deg s  = T$; 
    \item $\mathcal{U}$ as the set of products $u = \ell_1\ldots \ell_k$ of distinct irreducible monic polynomials $\ell_i$ with $\deg \ell_i =\fl{W/k}$, $i=1, \ldots, k$. 
\end{itemize}  
Note that any product of the form $suv$ with $(s,u,v) \in \mathcal{S} \times \mathcal{U}^2$ is $\alpha r$-smooth.

Fix some polynomial $a(X) \in \F_q[X]$ with $\gcd(F, a) = 1$. Let $N$ be the number of solutions to 
\begin{equation} \label{Nequation}
     suv \equiv a \pmod{F}, \qquad (s,u,v) \in \mathcal{S} \times \mathcal{U}^2.
\end{equation} 
As before, let $\cX_F$ denote the set of $q^r-1$  
multiplicative characters on the finite field $\F_q[X]/F(X) \cong \F_{q^r}$
and let  $\cX_F^*= \cX_F \setminus \{\chi_0\}$ be the set of nonprincipal characters.  By orthogonality, and then rearranging, we can write 
\begin{align*}
    N 
    &= \sum_{(s,u,v) \in \mathcal{S} \times \mathcal{U}^2}\frac{1}{q^r-1}\sum_{\chi \in  \cX_F}\chi(suva^{-1})\\
    &= \frac{1}{q^r-1}\sum_{\chi \in  \cX_F}\chi(a^{-1})\sum_{s \in \mathcal{S}}\chi(s)\(\sum_{u \in \mathcal{U}}\chi(u) \)^2.
\end{align*}
Now separating out the trivial character we get 

\begin{equation} \label{Nexpansion}
    N = \frac{\#\mathcal{S}(\#\mathcal{U})^2}{q^r-1} + \frac{1}{q^r-1}\sum_{\chi \in  \cX_F^*}\chi(a^{-1})\sum_{s \in \mathcal{S}}\chi(s)\(\sum_{u \in \mathcal{U}}\chi(u) \)^2.
\end{equation} 
We set 
$$R = \sum_{\chi \in  \cX_F^*}\chi(a^{-1})\sum_{s \in \mathcal{S}}\chi(s)\(\sum_{u \in \mathcal{U}}\chi(u)\)^2. $$
Since $n > 2$ we have $T < r$ so by Lemma~\ref{lem:SqFreeSum} we have
\begin{align*}
    |R| 
    &\leq q^{T/2 +o(r)}\sum_{\chi \in  \cX_F^*}\left|\sum_{u \in \mathcal{U}}\chi(u) \right|^2\\
    &\leq q^{T/2 +o(r)}\sum_{\chi \in  \cX_F}\left|\sum_{u \in \mathcal{U}}\chi(u) \right|^2 = q^{T/2 +o(r)}(q^r-1)\#\mathcal{U}.
\end{align*}
Hence substituting this back into~\eqref{Nexpansion} we derive
$$
    N = {\#\mathcal{S}(\#\mathcal{U})^2}q^{-r + o(1)} + O(q^{T/2 +o(r)}\#\mathcal{U}).
$$
Also we have 
$$\#S = q^{T + o(T)} = q^{T + o(r)} $$
and 
$$\#U = \binom{\varpi_{\fl{W/k}}}{k} = q^{W + o(W)} = q^{W + o(r)}. $$

Thus 
\begin{align*}
    N 
    &= q^{T + 2W -r + o(r)} + O(q^{T/2  + W+o(r)})\\
    &= q^{T + 2W -r + o(r)}\(1 + O(q^{-T/2-W +r}) \)\\
    &= q^{T + 2W -r + o(r)}\(1 + O(q^{-r\varepsilon}) \).
\end{align*}

Now we intend to show that for large enough $r$, this is strictly larger than the number of solutions with $suv$ not square-free. 

Suppose that some solution $suv$ is not squarefree. By construction it is divisible by the square of an irreducible monic $\ell$ with $\deg\ell= \fl{W/k}$. For a fixed $\ell$, this places the product $suv$ in a prescribed arithmetic progression modulo $\ell^2F$. Thus, there are at most 
$$O\(\frac{q^{T + 2k\fl{W/k}}}{q^{\deg\ell^2F}}\) = O\(q^{T + (2k-2)\frac{W}{k} - r}\)$$
 polynomials in any such progression. 
 We can say this, since  
 $$T + (2k-2)\frac{W}{k} \geq r(1+\epsilon) + r\beta\(1-\frac{2}{k}\) - 1 $$
 and for sufficiently large $r$, this is greater than $r$ because $1-2/k > 0$.
Now summing over all possible $\ell$,
$$\sum_{\ell \in \cP_{\fl{W/k}}}O\(q^{(2k-2)W/k + T - r}\)  = O\(q^{W(2-1/k) + T -r}\).$$
Finally, we note that a given product $suv$ corresponds to $q^{o(r)}$ triples $(s,u,v) \in \mathcal{S}\times \mathcal{U}^2$ (see~\cite[Lemma~1]{CillShp}), so we get at most
$q^{W(2-1/k) + T - r + o(r)} $
solutions that are not square-free. 
Thus for large enough $r$ at least one product $suv$ with 
$$\deg suv  \leq T + 2W \leq r\(\frac{1}{n}+ 2\beta\) = r\(2 - \frac{1}{n} + 2\varepsilon\) $$
satisfying~\eqref{Nequation} is square-free. 

Since $n$ can be chosen arbitrarily large and $\varepsilon$ can be chosen arbitrarily small, the result follows.

\subsection{Proof of Theorem~\ref{thm:Psi-LB}}
Let 
$$n =\fl{\(\frac{\alpha-\beta}{2}-\frac{\varepsilon}{2}\)r} \mand h = \fl{ \beta r}. $$

We wish to 
relate $\Psi^\#(k,m;F, a)$ and $N^\#_{F}(a, n, h)$. By construction we have $n \leq h \leq m$ and $2n + h \leq k$. Thus, it is clear that any triple $(\ell_1, \ell_2, u)$  satisfying~\eqref{eq:Ncongruence} with $\ell_1\ell_2u$ square-free corresponds to a unique $g$ satisfying the congruence~\eqref{eq:psisfcongruence}. The converse does not necessarily hold, but if it does hold for some $g$ then there are $q^{o(r)}$ such triples that correspond to it. Thus we can write 
$$
     \Psi^\#(k,m;F, a) \geq N^\#_{F}(a, n, h) q^{o(r)} - T q^{o(r)}, 
$$
where $T$ is an upperbound for the number of triples with $\ell_1 = \ell_2$, $\ell_1 | u$ or $\ell_2|u$. We consider each case, in order to estimate $T$. If $\ell_1 = \ell_2$, then $u$ is uniquely determined so there are $O(q^{n})$ triples in this case. If $\ell_1 |u$, then we write $u = \ell_1v$ and apply Lemma 2.9 but replacing $h$ with $h-n$. The same argument applies if $\ell_2|u$. Thus we have 
\begin{align*}
     \Psi^\#(k,m;F, a) 
     &\geq q^{o(r)}N^\#_{F}(a, n, h) + O \((q^{h-r} + 1)q^{n + o(r)}\)\\
     &\geq q^{o(r)}N^\#_{F}(a, n, h) + O \(q^{n + o(r)}\).
\end{align*}

Now we can apply Lemma~\ref{lem:Nboundsqfree} to estimate the leading term. Recalling~\eqref{eq:CountIrred},  we see that asymptotically we have 
$$\varpi_n^2q^{h-r-1} \sim \frac{q^{2n+h-r-1}}{n^2} = q^{2n+h-r + o(r)}. $$
Now we let $d = \fl{r\varepsilon/2}$. 
Then for large enough $r$ we can say 
\begin{align*}
    &\Psi^\#(k,m;F, a) \\
     &\hspace{2em} \geq 
     q^{2n+h-r + o(r)} + O\(\(q^{2n+h-d-r} + q^{B(r,n) + d} + q^{h/2} + q^{n}\)q^{o(r) }\)\\
     &\hspace{2em} = q^{r(\alpha - 1- \varepsilon) + o(r)} + O\(q^{B(r,n) + d + o(r)}\).
\end{align*}
These simplifications have been made as $B(r,n)$ dominates $n$, and the main term dominates $q^{2n+h-d-r}$. Also, the main term dominates $q^{h/2}$ since $\alpha-1>9/2-3\beta-1>\beta/2$ for $\beta \leq 1$.

Next, it remains to show that the main term always dominates the error term $q^{B(r,n) + d + o(r)}$. We split the discussion into two cases. 

Firstly, suppose $\alpha \in (9/2 - 3\beta, 2/3 + \beta]$. Since $\alpha \leq 2/3 +\beta$, we have $n < r/3$ which means $B(r,n) = 3n/2 + r/8$ so 
$$
    \Psi^\#(k,m;F, a) 
      \geq q^{r(\alpha - 1- \varepsilon) + o(r)} + O\(q^{3r(\alpha - \beta)/4 - r\varepsilon/4 + r/8 +  o(r)}\).
$$
For $\varepsilon$ sufficiently small, we have $\alpha > 9/2 - 3\beta + 7\varepsilon $, and this gives 
$$\alpha-1-\varepsilon - (3(\alpha-\beta)/4 - \varepsilon/4 + 1/8) > \varepsilon, $$
so the main term dominates the error term. 

Secondly, suppose $\alpha \in (2/3 + \beta, 3\beta]$. Then we have 
$$2/3 + \beta + \varepsilon < \alpha \leq 3\beta < 2 + \beta + \varepsilon $$
for $\varepsilon$ sufficiently small. This means $r/3 \leq n < r$, meaning that $B(r,n) = 15n/8$. Therefore, 
$$
    \Psi^\#(x,y;F, a) \geq q^{r(\alpha - 1- \varepsilon) + o(r)} 
    + O\(q^{15r(\alpha-\beta)/16  -7r\varepsilon/16 + o(r)}\).
$$
Now for $\varepsilon$ sufficiently small we have $\alpha > 2/3 + \beta + 25\varepsilon$, and recalling that $\beta > 23/24$ we get 
$$\alpha - 1- \varepsilon - (15(\alpha-\beta)/16  -7\varepsilon/16) > \varepsilon $$
and thus the main term dominates the error term.

Therefore, in every case we conclude 
$$
    \Psi^\#(k,m;F, a) \geq q^{r(\alpha - 1- \varepsilon) + o(r)}.
$$
After noting that $\varepsilon$ can be arbitrarily small, the result follows.

 \section*{Acknowledgements}
During the preparation of this work, C.B. was supported by an Australian Government Research Training Program (RTP) Scholarship and 
I.E.S. by the Australian Research Council Grant DP170100786.

\bibliographystyle{plain}


\end{document}